\documentclass[11pt]{amsart}
\usepackage[vmargin=2.5cm,hmargin=2.5cm]{geometry}
\usepackage{enumerate}
\usepackage{color}
\usepackage{url}

\newtheorem{theorem}{Theorem}
\newtheorem{lemma}[theorem]{Lemma}
\newtheorem{corollary}[theorem]{Corollary}
\newtheorem{proposition}[theorem]{Proposition}

\newtheorem{defn}[theorem]{Definition}

\theoremstyle{remark}
\newtheorem{example}[theorem]{Example}
\newtheorem{remark}[theorem]{Remark}

\newcommand{\peb}[1]{\left\lfloor {#1}\right\rfloor}

\title{Delta sets for symmetric numerical semigroups with embedding dimension three} 

\author{P. A. Garc\'{\i}a-S\'{a}nchez}
\address{Departamento de \'Algebra and IEMath-GR,
   Universidad de Granada, 18071 Granada, Espa\~na}
\email{pedro@ugr.es}

\author{D. Llena}
\address{Departamento de Matem\'aticas, Universidad de
   Almer\'{\i}a, 04120 Almer\'{\i}a, Espa\~na}
\email{dllena@ual.es}

\thanks{The first and second author are supported by the project FQM-343, MTM2014-55367-P, and FEDER funds}

\author{A. Moscariello}
\address{Dipartimento di Matematica, Università di Pisa, Largo Bruno Pontecorvo 5, 56127 Pisa, Italy.}

\email{moscariello@mail.dm.unipi.it}


\begin{document}
\maketitle
\begin{abstract}
	This work extends the results known for the Delta sets of non-symmetric numerical semigroups with embedding dimension three to the symmetric case. Thus, we have a fast algorithm to compute the Delta set of any embedding dimension three numerical semigroup. Also, as a consequence of these resutls, the sets that can be realized as Delta sets of numerical semigroups of embedding dimension three are fully characterized.
\end{abstract}
\section{Introduction}
Delta sets (or sets of distances) were first introduced in \cite{G} as a tool to study factorizations in non-unique factorization domains. Since then several authors have studied their properties. In particular, for numerical semigroups, in \cite{BCKR}, the firsts results for some special cases of embedding dimension three numerical semigroups are presented. Delta sets for numerical semigroups are eventually periodic as explained in \cite{CHK}, and thus, if a bound for this periodicity is known, the whole Delta set of a numerical semigroup can be computed. The bounds given in \cite{CHK} were improved in \cite{GMV}. In \cite{MOP} a dynamical procedure to compute Delta sets for numerical semigroups is presented, that makes use of the bound given in \cite{GMV}; this procedure has been implemented in \texttt{numericalsgps} \cite{numericalsgps}. Recently, a new procedure to compute Delta sets of any affine semigroup using G\"obner basis has been presented in \cite{G-SOW}. Needless to say, all these algorithms and bounds were the consequence of a better (theoretical) understanding of Delta sets. 

In \cite{CGLMS} it is shown that the maximum of the Delta set of a numerical semigroup is attained in a Betti element of the semigroup (indeed, it is shown that this holds for a wider class of atomic monoids). This does not provide a way to compute the whole set. The minimum was known to be the greatest common divisor of the Delta set since \cite{G}; however, the elements in the interval determined by this minimum and maximum element of the Delta set are not known in general. Some realization results were given in \cite{k-deltas}, while in \cite{G-SLM} the sets that can be realized as the Delta set of a non-symmetric numerical semigroup with embedding dimension three are completely characterized. In that paper, the authors present a procedure that strongly reduces the time needed to compute the Delta sets of non-symmetric embedding dimension three numerical semigroups. In this manuscript, we extend this algorithm to the symmetric case. 

Although the main results yield the same algorithm for both symmetric and non-symmetric numerical semigroups with embedding dimension three, there are significant differences in the intermediate results and their proofs. These differences stem from the different structure of the Betti elements of these semigroups (see Proposition \ref{struct}). 

The factorizations of the Betti elements of non-symmetric embedding dimension three numerical semigroups have been thoroughly studied, and thus provided an \emph{easier} starting point. However, in that case, some technical and tedious detours were needed to obtain our result.

Due to the fact that, in the symmetric case, we do not have a unique minimal presentation (in contraposition to what happens in the non-symmetric case), the first step is to choose the right factorization of the Betti elements to start our algorithm. However, in this case we can work directly on the Euclid's set (see Definition \ref{Euclidset}), whereas this cannot be done in the non-symmetric case without building a certain theoretical framework. In this way, we can better see the idea behind the main result, and the technical results are no longer necessary. Moreover, the proofs and the reading of the paper are more comfortable, and the number of pages is reduced.  Unfortunately,
from our reasoning, it emerged that there are no relations between these two cases; we think that it is not possible to find a common way to deduce both settings together.

This paper, even though is complementary to \cite{G-SLM}, is self-contained, and therefore can be read separately from \cite{G-SLM}. After some preliminary results, we explain how we can choose the elements $\delta_1$ and $\delta_2$ needed to start the algorithm. As we mentioned above, this was not needed in the non-symmetric case, since in that setting the Delta sets of the Betti elements of the numerical semigroup are singletons. Then, we prove our main result, which yields an algorithm that works in the same way as in the non-symmetric case; in particular, notice that Example \ref{mainexample} provided in this paper gives the same Delta set as in \cite[Example 38]{G-SLM}, because it is obtained from the same $\delta_1$ and $\delta_2$.

\section{Preliminaries}
Let $\mathbb N$ be the set of non negative integers.
Take $n_1$, $n_2$, $n_3\in \mathbb N$ with $\gcd(n_1,n_2,n_3)=1$, and define $S$ as the numerical semigroup minimally generated by $\{n_1,n_2,n_3\}$, that is,
\[
S=\{a_1n_1+a_2n_2+a_3n_3 \mid a_1,a_2,a_3\in \mathbb N\}.
\] 
The condition $\gcd(n_1,n_2,n_3)=1$ is equivalent to say that the set of gaps, $\mathrm G(S)=\mathbb N\setminus S$, has finitely many elements. The maximum of $\mathbb Z\setminus S$ is called the Frobenius number, and we denote it as $\mathrm F(S)$.

\begin{defn}
	A numerical semigroup $S$ is symmetric if $x\in \mathbb Z\setminus S$ implies $\mathrm F(S)-x\in S$.
\end{defn}
The reader interested in numerical semigroups may have a look at \cite{RG} and for some applications to \cite{AG-S}. 

The \emph{set of factorizations} of $s\in S$ is \[\mathsf Z(s)=\{(z_1,z_2,z_3)\in \mathbb N^3 \mid z_1n_1+z_2n_2+z_3n_3=s\}.\] We denote the \emph{length} of a factorization $\mathbf z=(z_1,z_2,z_3)\in \mathsf Z(s)$ as \[\ell(\mathbf z)=z_1+z_2+z_3.\] The \emph{set of lengths} of $s\in S$ is 
\[\mathsf L(s)=\{\ell(\mathbf z)\mid \mathbf z\in\mathsf Z(s)\}.\] It is easy to see that $\mathsf L(s) \subset [0,s]$, and consequently $\mathsf L(s)$ is finite. So, it is of the form $\mathsf L(s)=\{l_1,\ldots ,l_k\}$ for some positive integers $l_1<l_2<\cdots <l_k$.  The set
\[
\Delta(s)=\{l_i-l_{i-1} \mid 2\le i\le k\}
\]
is known as the \emph{Delta set} of $s \in S$ (or set of distances), and 
the \emph{Delta set} of $S$ is defined as
\[\Delta(S)=\cup_{s \in S}\Delta(s).\]

Now, we define 
\[
M_S=\{ \mathbf v=(v_1,v_2,v_3)\in \mathbb Z^3\mid v_1n_1+v_2n_2+v_3n_3=0\}.
\]


We can extend the $\ell$ function to elements in $M_S$: for $\mathbf v=(v_1,v_2,v_3)\in \mathbb Z^3$, set 
\[
\ell(\mathbf v)=v_1+v_2+v_3.
\]
Note that $\ell$ is a linear map.

The aim of this paper is to compute the Delta set of $S$ by using Euclid's algorithm with two special elements in this set. These special distances are associated to particular elements of $S$, called the Betti elements. The definition of Betti element relies in the construction of a graph associated to the elements in the semigroup. Let $s\in S$. The graph $\nabla_s$ is the graph with vertices $\mathsf Z(s)$, the set of factorization of $s$, and $\mathbf z\mathbf z'$ is an edge if if $\mathbf z\cdot\mathbf z'\neq 0$, that is, there exists a common nonzero coordinate in both factorizations.

\begin{defn}
	An element $s\in S$ is called a Betti element if and only if $\nabla_s$ is not connected. The set of Betti elements will be denoted by $\mathrm{Betti}(S)$.
\end{defn}

A numerical semigroup with embedding dimension three might have up to three Betti elements. In \cite[Example 8.23]{RG} it is shown that $\mathrm{Betti}(\langle n_1,n_2,n_3\rangle=\{ c_1n_1,c_2n_2,c_3n_3\}$, where $c_i$ is the least positive integer such that $c_in_i\in \langle n_j,n_k\rangle$, $\{i,j,k\}=\{1,2,3\}$. 

\begin{proposition}\cite[Chapter 9]{RG}\label{struct}
	Let $S$ an embedding dimension three numerical then $1\le \sharp \mathrm{Betti}(S)\le 3$. Moreover, $S$ is non-symmetric if and only if  $\sharp \mathrm{Betti}(S)=3$. 
\end{proposition}

\begin{example}
Consider $S=\langle 3,5,7\rangle=\{0,3,5,\to\}$ (the arrow means that all integers greater than $5$ are in the semigroup). 
In this case,  $\mathrm{Betti}(S)=\{10,12,14\}$, and 
		\[
		\begin{array}{rclrclrcl}
		\mathsf Z(10) & = & \{(1,0,1),(0,2,0)\}, & \mathsf{L}(10) & = & \{2\},  & \Delta(10) & = & \emptyset, \\
		\mathsf Z(12) & = & \{(4,0,0),(0,1,1)\}, & \mathsf{L}(12) & = & \{2,4\},  & \Delta(12) & = & \{2\},\\
		\mathsf Z(14) & = & \{(3,1,0),(0,0,2)\}, & \mathsf{L}(14) & = & \{2,4\}, & \Delta(14) & = & \{2\}.
		\end{array}
		\]

Let $S=\langle 6,8,11\rangle$, with Betti elements $22$ and $24$.
	\[
	\begin{array}{rclrclrcl}
	\mathsf Z(22) & = & \{(1,2,0),(0,0,2)\}, & \mathsf{L}(22) & = & \{2,3\}, & \Delta(22) & = & \{1\},\\
	\mathsf Z(24) & = & \{(4,0,0),(0,3,0)\}, & \mathsf{L}(24) & = & \{3,4\}, & \Delta(24) & = & \{1\}.
	\end{array}
	\]
Finally, let $S=\langle 6,10,15\rangle$. Then $\mathrm{Betti}(S)=\{30\}$, 
\[
\mathsf Z(30)=\{(5,0,0),(0,3,0),(0,0,2)\},\  \mathsf L(30)=\{2,3,5\},\  \mathsf \Delta(30)=\{1,2\}.
\]
\end{example} 
More details on the relation between Betti elements and Delta sets can be found in \cite{CGLMS}.

It is straightforward to see that from the factorizations of every $b\in \mathrm{Betti}(S)$, we can construct $\{\mathbf z-\mathbf z'\in M_S\mid \mathbf z\neq \mathbf z', \mathbf z,\mathbf z'\in \mathsf Z(b), b\in \mathrm{Betti}(S)\}$, which generates $M_S$ as a group. 

For instance, for $S=\langle 6,8,11\rangle$, $(1,2,-2)$ and $(4,-3,0)$ generate as a group the set of integer solutions of the equation $6x_1+8x_2+11x_3=0$.  

Since the structure of the Delta set of a non-symmetric numerical semigroup is known, we focus on the study of the symmetric numerical semigroups.

\section{Two cases in the symmetric setting}\label{sec:uno}
Let $S$ be a numerical semigroup minimally generated by $\{n_1,n_2,n_3\}$, and assume that $S$ is symmetric. In this Section, we define the integers and vectors that allow us to study $M_S$ and construct $\Delta(S)$. 

By Proposition \ref{struct}, $\#\mathrm{Betti}(S) \le 2$. Now, we going to choose a suitable basis $\{\mathbf v_1, \mathbf v_2 \}$ for $M_S$ according to $\#\mathrm{Betti}(S)$.

\begin{enumerate}[(1)]
	\item \textbf{A single Betti element.}
	
	\noindent If $\#\mathrm{Betti}(S)=1$, then $(n_1,n_2,n_3)=(s_2s_3,s_1s_3,s_1s_2)$ for some positive pairwise coprime integers $s_1>s_2>s_3$ (see \cite{single}). Moreover, $\mathrm{Betti}(S)=\{s_1s_2s_3\}$, and the set of factorizations of $s_1s_2s_3$ is  $\mathsf Z(s_1s_2s_3)=\{(s_1,0,0),(0,s_2,0),(0,0,s_3)\}$.
	In this setting, it is easy to see that $M_S$ is the group spanned by $\{(s_1,-s_2,0),(0,s_2,-s_3)\}$. We set 
	$\mathbf v_1=(s_1,-s_2,0)$, $\mathbf v_2=(0,s_2,-s_3)$.
	
	\item \textbf{Two Betti elements.}
	
	\noindent If $\#\mathrm{Betti}(S)=2$, we have $(n_1,n_2,n_3)=(am_1,am_2,bm_1+cm_2)$ with $a\ge 2$ and $b+c\ge 2$, and we can also assume that $m_1<m_2$ \cite[Theorem 10.6]{RG}. In this setting \[\mathrm{Betti}(S)=\{a(bm_1+cm_2), am_1m_2\},\]
	with 
	\begin{align*}
	    \mathsf Z(a(bm_1+cm_2))=\Big\{ & \left(b-\left\lfloor\tfrac{b}{m_2}\right\rfloor m_2,c+\left\lfloor\tfrac{b}{m_2}\right\rfloor m_1,0\right),\ldots,(b-m_2,c+m_1,0),(b,c,0), \\ & (b+m_2,c-m_1,0)\ldots, \left(b+\left\lfloor\tfrac{c}{m_1}\right\rfloor m_2,c-\left\lfloor\tfrac{c}{m_1}\right\rfloor m_1,0\right),(0,0,a)\Big\}
	\end{align*} 
	and $\mathsf Z(am_1m_2)=\{(m_2,0,0),(0,m_1,0)\}$. Moreover, $M_S$ is spanned by $\{(m_2,-m_1,0),(b+\lambda m_2,c-\lambda m_1,-a)\}$ 
	for any $\lambda \in \mathbb Z$. We choose 
	$\lambda\in\left\{-\left\lfloor{b}/{m_2}\right\rfloor,\ldots,\left\lfloor{c}/{m_1}\right\rfloor\right\}$ such that $\ell((b+\lambda m_2,c-\lambda m_1,-a))$ is minimal. We define $\mathbf v_1=(m_2,-m_1,0)$ and $\mathbf v_2= (b+\lambda m_2,c-\lambda m_1,-a)$.  
\end{enumerate}

 Now, we define $\delta_1=\ell(\mathbf v_1)$ and $\delta_2=|\ell(\mathbf v_2)|$.  We consider the absolute value for $\ell(\mathbf v_2)$, since it might happen that $\ell(\mathbf v_2)<0$ when there are two Betti elements and  $a>b+c+\lambda(m_2-m_1)$. In order to keep trace of the sign of $\ell(\mathbf v_2)$, let $\mathrm{sgn}$ be the sign function, and set $\sigma=\mathrm{sgn}(\ell(\mathbf v_2))$. In this way, we have $\delta_2=\sigma\ell(\mathbf v_2)$.  
 
 The integers $\delta_1,\delta_2$ just defined are tightly related to $\Delta(S)$.

\begin{proposition}\label{max-delta}
Let $S$ be a symmetric numerical semigroup, and let $\delta_1$ and $\delta_2$ be defined as above. Then \[\max\Delta(S)=\max\{\delta_1,\delta_2\}.\] 
\end{proposition}
\begin{proof}
We know from \cite{CGLMS} that $\max\Delta(S)=\max\{\max\Delta(b) \mid b\in \mathrm{Betti}(S)\}$. 

If $\#\mathrm{Betti}(S)=1$, since $\mathsf Z(s_1s_2s_3)=\{(s_1,0,0),(0,s_2,0),(0,0,s_3)\}$, we get $\Delta(s_1s_2s_3)=\{s_1-s_2,s_2-s_3\}=\{\delta_1,\delta_2\}$, and the thesis follows. 

If $\#\mathrm{Betti}(S)=2$, we need to study $\mathsf Z(b)$ with $b\in \mathrm{Betti}(S)=\{a(bm_1+cm_2), am_1m_2\}$. 
\begin{enumerate}[$\bullet$]
	\item $\mathsf Z(am_1m_2)=\{(m_2,0,0),(0,m_1,0)\}$. Hence $\Delta(am_1m_2)=\{\delta_1\}$.
	
	\item $\mathsf Z(a(bm_1+cm_2))=\{ (0,0,a)\} \cup \left\{(b+k m_2,c-k m_1,0)\mid k \in \left\{-\left\lfloor{b}/{m_2}\right\rfloor,\ldots,\left\lfloor{c}/{m_1}\right\rfloor\right\}\right\}$.
\end{enumerate}  
Observe that $\mathsf L(a(bm_1+cm_2))=\{a\}\cup \left\{ b+c+k(m_2-m_1)\mid k \in \left\{-\left\lfloor{b}/{m_2}\right\rfloor,\ldots,\left\lfloor{c}/{m_1}\right\rfloor\right\}\right\}$. Consider the following cases.
\begin{enumerate}[(1)]
	\item If $a\le b+c-\lfloor b/m_2\rfloor (m_2-m_1)$, then $\lambda=-\lfloor b/m_2\rfloor$ and $\delta_2=b+c-\lfloor b/m_2\rfloor (m_2-m_1)-a$. Hence $\Delta(a(bm_1+cm_2))$ equals $\{\delta_1,\delta_2\}$, or $\{\delta_1\}$ if $a=b+c-\lfloor b/m_2\rfloor (m_2-m_1)$ (that is, $\delta_2=0$).
	
	\item If $a\ge b+c+\lfloor c/m_1\rfloor (m_2-m_1)$, then $\lambda=\lfloor c/m_1\rfloor$, and we argue as in the previous case.
	
	\item Finally, if $b+c-\lfloor b/m_2\rfloor (m_2-m_1) < a < b+c+\lfloor c/m_1\rfloor (m_2-m_1)$, there exists $k\in \left\{-\left\lfloor{b}/{m_2}\right\rfloor,\ldots,\left\lfloor{c}/{m_1}\right\rfloor -1\right\}$ such that $b+c+k(m_2-m_1) \le a <b+c+(k+1)(m_2-m_1)$. Then $\lambda$ is either $k$ or $k+1$, and $\Delta(a(bm_1+cm_2))=\{\delta_1,\delta_2,|\delta_2-\delta_1|\}$ (unless $\delta_2=0$ and $\Delta(a(bm_1+cm_2))=\{\delta_1\}$). 
\end{enumerate}
In any case, $\max\Delta(S)=\max\{\delta_1,\delta_2\}$.
\end{proof}

The following result is a particular instance of a more general property. 

\begin{proposition}\label{min-delta}\cite[Corollay 3.1]{CGLMS}
	Let $S$ be a symmetric numerical semigroup. Then \[\min\Delta(S)=\gcd\{\delta_1,\delta_2\}.\] 
\end{proposition}

Thus, $\max \Delta(S)$ is either $\delta_1$ or $\delta_2$, while in our setting $\min \Delta(S)=\gcd(\delta_1,\delta_2)$, and each element of $\Delta(S)$ is a multiple of this greatest common divisor \cite{G}. 

\begin{remark}
Observe that, in both cases, the vector $\mathbf v_1$ has the first coordinate positive, the second one negative, and the third coordinate equal to zero. Our choice of $\lambda$ in the case $\#\mathrm{Betti}(S)=2$ ensures that the vector $\mathbf v_2$ has first coordinate nonnegative, positive the second, and the third coordinate negative. We will represent this fact as: 
\[\mathbf v_1=(+,-,0) \hbox{ and }\mathbf v_2=(+,+,-).\]
\end{remark}

\begin{proposition}\label{signos-v-deltas}
Let $S$ be a symmetric numerical semigroup with embedding dimension three. Let $\delta_1$, $\delta_2$, $\mathbf v_1$, $\mathbf v_2$ and $\sigma$ be defined as above. 
\begin{enumerate}[(1)]
\item if $\sigma=1$, we have that $\delta_2\mathbf v_1-\sigma\delta_1\mathbf v_2=(?,-,+)$.

\item if $\sigma=-1$, we have that $\delta_2\mathbf v_1-\sigma\delta_1\mathbf v_2=(+,?,-)$.
\end{enumerate}
The symbol ``$\ ?$'' denotes the sign of this coordinate is not determined.
\end{proposition}
\begin{proof}
If $\sigma=1$ and $\#\mathrm{Betti}(S)=1$, we know that $\mathbf v_1=(s_1,-s_2,0)$ and $\mathbf v_2=(0,s_2,-s_3)$. Then
\[
\delta_2\mathbf v_1-\delta_1\mathbf v_2=( \delta_2s_1,-\delta_2s_2-\delta_1s_2,\delta_1s_3)=(\delta_2s_1,-s_2(\delta_2+\delta_1),\delta_1s_3)=(+,-,+).
\]
For $\sigma=1$ and $\#\mathrm{Betti}(S)=2$, we have $\mathbf v_1=(m_2,-m_1,0)$ and $\mathbf v_2=(b+\lambda m_2,c-\lambda m_1,-a)$, whence
\[
\delta_2\mathbf v_1-\delta_1\mathbf v_2=( \delta_2m_2-\delta_1(b+\lambda m_2),-\delta_2m_1-\delta_1(c-\lambda m_1),\delta_1a)=(?,-,+).
\]
Finally if $\sigma=-1$ and $\#\mathrm{Betti}(S)=2$, 
\[
\delta_2\mathbf v_1+\delta_1\mathbf v_2=( \delta_2m_2+\delta_1(b+\lambda m_2),-\delta_2m_1+\delta_1(c-\lambda m_1),-\delta_1a)=(+,?,-).\qedhere
\]
\end{proof}

\section{Euclid's set and the Delta set}

Associated to $\delta_1$ and $\delta_2$ we are going to define its Euclid's set as the set of all integers that appear in the naive implementation of the greatest common divisor algorithm. We will see that precisely this set is the Delta set of the semigroup except to the zero element. 

\begin{proposition}\label{unique}
Let $\delta_1,\delta_2$ be two positive integers, and let $x\in \{1,\ldots, \max\{\delta_1,\delta_2\}\}\cap g\mathbb Z$, where $g=\gcd(\delta_1,\delta_2)$. Then there exist unique $(x_1,x_2)$ and $(x_1',x_2')$ in $\mathbb Z^2$ such that 
\begin{enumerate}
\item $x=x_1 \delta_1+x_2\delta_2$ with $-\delta_1/g<x_2 \le 0< x_1 \le \delta_2/g$,
\item $x=x_1'\delta_1+x_2'\delta_2$ with $-\delta_2/g<x_1' \le 0< x_2' \le \delta_1/g$.
\end{enumerate}
Moreover $(x_1,x_2)=(x_1'+\delta_2/g,x_2'-\delta_1/g)$.
\end{proposition}

We will denote 
\begin{itemize}
    \item $x^{(\delta_1,\delta_2)}=(x_1,x_2)$, where $0<x_1 \le \delta_2/g$, $-\delta_1/g<x_2\le 0$, and $x=x_1\delta_1+x_2\delta_2$, 
    \item $x'^{(\delta_1,\delta_2)}=(x'_1,x'_2)$,  where $-\delta_2/g<x'_1 \le 0$, $0<x_2\le \delta_1/g$, and $x=x_1'\delta_1+x_2'\delta_2$.   
\end{itemize}
In particular, $\delta_1^{(\delta_1,\delta_2)}=(1,0)$, $\delta_1'^{(\delta_1,\delta_2)}= (1-\delta_2/g,\delta_1/g)$, $\delta_2'^{(\delta_1,\delta_2)}=(0,1)$  and $\delta_2^{(\delta_1,\delta_2)}=(\delta_2/g,1-\delta_1/g)$.

We want to depict the set of elements obtained after applying Euclid's greatest common divisor algorithm to $\delta_1$ and $\delta_2$. We will use the naive approach that uses substraction instead of remainders of divisions. Our set will be decomposed in subsets corresponding to the algorithm using remainders. 

\begin{defn}\label{Euclidset}
	Let $\delta_1$ and $\delta_2$ be positive integers, and define $\eta_1=\max\{\delta_1,\delta_2\}$, $\eta_2=\min\{\delta_1,\delta_2\}$ and $\eta_3=\eta_1\bmod \eta_2$. In general, for $j>0$, define
	$\eta_{j+2}=\eta_j-\peb{\frac{\eta_j}{\eta_{j+1}}}\eta_{j+1}=\eta_j\bmod \eta_{j+1}$. Let $i$ be the maximum index such that $\eta_{i+1}>0$. Define
	\[
	\begin{array}{l}
	\mathrm D(\eta_1,\eta_2)=\{\eta_1,\eta_1-\eta_2,\ldots ,\eta_1\bmod \eta_2=\eta_3\}, \\
	\mathrm D(\eta_2,\eta_3)=\{\eta_2,\eta_2-\eta_3,\ldots ,\eta_2\bmod \eta_3=\eta_4\}, \\
	\mathrm D(\eta_3,\eta_4)=\{\eta_3-\eta_4,\ldots ,\eta_3\bmod \eta_4=\eta_5\}, \\
	\cdots \\
	\mathrm D(\eta_{j-1},\eta_j)=\{\eta_{j-1}-\eta_j,\ldots ,\eta_{j-1}\bmod \eta_j=\eta_{j+1}\}, \\
	\cdots \\
	\mathrm
	D(\eta_i,\eta_{i+1})=\{\eta_i-\eta_{i+1},\ldots ,\eta_i\bmod \eta_{i+1}=\eta_{i+2}=0\}. \\
	\end{array}
	\]
	Let us denote by $\mathrm{Euc}(\delta_1,\delta_2)=\bigcup_{i\in I}\mathrm D(\eta_i,\eta_{i+1})$, where $I=\{i\in\mathbb N \mid \eta_{i+1}>0\}$. 
\end{defn}
Observe that $\eta_1\in \mathrm D(\eta_1,\eta_2)$ and $\eta_2\in \mathrm D(\eta_2,\eta_3)$, but this is no longer the case for $i>2$. This is because we want the union in the definition of $\mathrm{Euc}(\delta_1,\delta_2)$ to be disjoint. Also, $\eta_j\in \mathrm D(\eta_{j-2},\eta_{j-1})$ as $\eta_j=\eta_{j-2}-\lfloor \eta_{j-2}/\eta_{j-1}\rfloor\eta_{j-1}$ for any $j>2$.

Notice also that $\mathrm{Euc}(\delta_1,\delta_2)\subset g\mathbb Z$ where $g=\gcd(\delta_1,\delta_2)$.

\begin{example}\label{mainexample-pre}
Let $S=\langle s_2s_3,s_1s_3,s_1s_2\rangle$, with $s_1=548$, $s_2=155$, and $s_3=13$. Then $\mathbf v_1=(548,-155,0)$ and $\mathbf v_2=(0,155,-13)$. Hence
$\delta_1=393$ and $\delta_2=142$. In this setting, 

\begin{itemize}
\item $\mathrm D(393,142)=\{393, 251, 109\}$,
\item $\mathrm D(142,109)=\{142,33\}$,
\item $\mathrm D(109,33)=\{76,43,10\}$, 
\item $\mathrm D(33,10)=\{23,13,3\}$,
\item $\mathrm D(10,3)=\{7,4,1\}$,
\item $\mathrm D(3,1)=\{2,1,0\}$.
\end{itemize}
Then $\mathrm{Euc}(\delta_1,\delta_2)=\{0,1,2,3,4,7,10,13,23,33,43,76,109,142,251,393\}$.
\end{example}

\begin{remark}\label{nota}
It is clear that $\eta_1>\eta_2>\cdots >\eta_i>\eta_{i+1}>\eta_{i+2}=0$, and  for $\eta_i^{(\eta_1,\eta_2)}=(\eta_{i1},\eta_{i2})$ and $\eta_{i+1}'^{(\eta_1,\eta_2)}=(\eta'_{i+1\, 1},\eta'_{i+1\, 2})$, the inequalities $\eta'_{i+1\, 1}\eta_{i1}\le 0$ and $\eta'_{i+1\, 2}\eta_{i2}\le 0$ hold. From these, we have too that, for $q>0$, $|\eta_{i1}|\le |\eta_{i1}-q\eta'_{i+1\, 1}|$ and $|\eta_{i2}|\le |\eta_{i2}-q\eta'_{i+1\, 2}|$ hold. Analogously, considering  $\eta_i'^{(\eta_1,\eta_2)}=(\eta'_{i1},\eta'_{i2})$ and $\eta_{i+1}^{(\eta_1,\eta_2)}=(\eta_{i+1\, 1},\eta_{i+1\, 2})$, we have $|\eta'_{i1}|\le |\eta'_{i1}-q\eta_{i+1\, 1}|$, $|\eta'_{i2}|\le |\eta'_{i2}-q\eta_{i+1\, 2}|$ for $q>0$. Observe that we have an equality only when one of the $\eta_{hk}$ is zero.

As $\eta_1^{(\eta_1,\eta_2)}=(1,0)$ and $\eta_2'^{(\eta_1,\eta_2)}=(0,1)=(-,+)$, we obtain $\eta_3^{(\eta_1,\eta_2)}=(1,-\peb{\eta_1/\eta_2})=(+,-)$ and $\eta_4'^{(\eta_1,\eta_2)}=(-\peb{\eta_2/\eta_3}$, $1+\peb{\eta_2/\eta_3}\peb{\eta_1/\eta_2})=(-,+)$, and so on.

This means that in each step in $\mathrm{Euc}(\delta_1,\delta_2)$, the absolute value of the $(\eta_1,\eta_2)$-coordinates increases, (and is not decreasing only when $x_{hk}=x'_{21}$). We will use this fact in the next propositions.
 \end{remark}

\begin{proposition}\label{premain}
Let the notation and hypotheses be as in Definition \ref{Euclidset}, and let $d\in\{\eta_3,\eta_3+g,\ldots ,\eta_1\}$. Then $d\in \mathrm D(\eta_1,\eta_2)$ if and only if $d_1=1$, with $(d_1,d_2)=d^{(\eta_1,\eta_2)}$ (according to the notation given in Proposition \ref{unique}).
\end{proposition}

\begin{proof}
Let $d\in\{\eta_3, \eta_3+g,\ldots, \eta_1\}$. If $d\in \mathrm D(\eta_1,\eta_2)$, then $d=\eta_1-k\eta_2$ for some $k\in \{0,\ldots,\lfloor \eta_1/\eta_2\rfloor\}$. Hence $d^{(\eta_1,\eta_2)}=(1,-k)$. For the converse, if $0<d=\eta_1+k\eta_2\le \eta_1$ for some $k\in \mathbb Z$, we obtain $-\eta_1<k\eta_2\le 0$, whence $- \eta_1/\eta_2< k\le 0$. Now, as $g=\gcd(\delta_1,\delta_2)\leq \eta_2$ we have $-\eta_1/g\leq - \eta_1/\eta_2<k$  and by Proposition \ref{unique}, $k$ will be equal to $d_2$, and consequently $d\in \mathrm D(\eta_1,\eta_2)$.
\end{proof}

\begin{proposition}\label{main}
Let the notation and hypotheses be as in Definition \ref{Euclidset}. 
	\begin{enumerate}[(1)]
		\item Let $x\in \{\eta_3,\eta_3+g,\ldots, \eta_1\}\setminus (\mathrm D(\eta_1,\eta_2)\cup \mathrm D(\eta_2,\eta_3))$.
		     \begin{enumerate}[(a)]
		     \item If $x^{(\eta_1,\eta_2)}=(x_1,x_2)=(+,-)$, then there exists an integer $d<x$, $d\in \mathrm D(\eta_1,\eta_2)$, such that $d^{(\eta_1,\eta_2)}=(d_1,d_2)$ with $0<d_1<x_1$ and $x_2<d_2<0$. 
		     \item If $x'^{(\eta_1,\eta_2)}=(x'_1,x'_2)=(-,+)$, then there exists an integer $d<x$, $d\in \mathrm D(\eta_2,\eta_3)$, such that $d'^{(\eta_1,\eta_2)}=(d_1',d_2')$ with $x_1'\le d_1'\le 0$ and $0<d_2'<x_2'$.
		     
		     The case $x'_1=d'_1=0$ corresponds to $x=d_2\eta_2$, that is, only when $x$ is a multiple of $\eta_2$.
	        \end{enumerate}
        \item In general, for any $x$ multiple of $g$ such that $x<\eta_3$, $x\notin \mathrm{Euc}(\delta_1,\delta_2)$, consider $j+1=\min\{k\in I\mid \eta_k<x\}$ and then $x\in \{\eta_{j+1},\eta_{j+1}+g,\ldots, \eta_{j}\}\setminus (\mathrm D(\eta_{j-1},\eta_j)\cup \mathrm D(\eta_j,\eta_{j+1})) = \{\eta_{j+1},\eta_{j+1}+g,\ldots, \eta_{j}\}\setminus \mathrm{Euc}(\delta_1,\delta_2)$.
        \begin{enumerate}[(a)]
        	\item If  $x^{(\eta_1,\eta_2)}=(x_1,x_2)=(+,-)$, then there exists an integer $d<x$, $d\in \mathrm D(\eta_{j-1},\eta_{j})\cup \mathrm D(\eta_{j},\eta_{j+1})$, such that $d^{(\eta_1,\eta_2)}=(d_1,d_2)$ with $0<d_1<x_1$ and $x_2<d_2<0$.
        	\item If $x'^{(\eta_1,\eta_2)}=(x'_1,x'_2)=(-,+)$, then there exists an integer $d<x$, $d\in \mathrm D(\eta_{j-1},\eta_{j})\cup \mathrm D(\eta_{j},\eta_{j+1})$, such that $d'^{(\eta_1,\eta_2)}=(d_1',d_2')$ with $x_1'<d_1'<0$ and $0<d_2'<x_2'$.
        \end{enumerate}
	\end{enumerate}
\end{proposition}
\begin{proof}
\begin{enumerate}[(1)]
\item We study each case independently.
\begin{enumerate}[(a)]
    \item As $\eta_1-\lfloor \frac{\eta_1}{\eta_2} \rfloor \eta_2= \eta_3 < x < \eta_1$ and $x\not\in \mathrm D(\eta_1,\eta_2)\cup\mathrm D(\eta_2,\eta_3)$, there exists an integer $k$ such that $\eta_1+k\eta_2<x<\eta_1+(k+1)\eta_2$ and $-\lfloor \frac{\eta_1}{\eta_2}\rfloor < k < 0$. Set $d=\eta_1+k\eta_2$. Then $d\in \mathrm D(\eta_1,\eta_2)$. Also, $\eta_2\ge g$ and $\delta_1\le \eta_1$, thus $-\frac{\delta_1}g\le -\frac{\eta_1}{\eta_2}\le \lceil -\frac{\eta_1}{\eta_2}\rceil= -\lfloor \frac{\eta_1}{\eta_2}\rfloor<k$. This implies that $d^{(\eta_1,\eta_2)}=(1,k)$. It is clear that $x_1\ge 2$ by Proposition \ref{premain}, and as  $d_1=1$ we have that $0<d_1<x_1$. Also, as  $x=x_1 \eta_1+x_2 \eta_2$ and $x_1\ge 2$, and we obtain that $2\eta_1+x_2 \eta_2\le x <\eta_1+(k+1)\eta_2$. Consequently, $x_2 \eta_2<\eta_1+x_2 \eta_2<(k+1)\eta_2$. Thus, $x_2<k+1$. Observe that $x_2\neq k$, since otherwise $x\ge 2\eta_1+k\eta_2> \eta_1+\eta_2+k\eta_2= \eta_1+(k+1)\eta_2$, a contradiction. Hence $x_2<k=d_2<0$.
    \item For $(x'_1,x'_2)$ we distinguish two cases.
     \begin{enumerate}[(i)]
      	\item If $x>\eta_2$, it suffices to take $d=\eta_2 \in \mathrm D(\eta_{2},\eta_{3})$. We have that $d_1'=0$ and $d_2'=1$. In light of Proposition \ref{unique}, it follows that $x_1'\le d_1'=0$ and $0<d_2'<x_2'$. It is clear that  $x'_2\neq 1$, since otherwise we would have  $x=x'_1\eta_1+x'_2\eta_2=x'_1\eta_1+\eta_2\le d'_1\eta_1+\eta_2=\eta_2$, that is, $x\le \eta_2$, a contradiction.
       	\item If $\eta_3<x<\eta_2$, consider $\mathrm D(\eta_2,\eta_3)=\{\eta_2,\eta_2-\eta_3,\eta_2-2\eta_3,\ldots ,\eta_4\}$. As $x\notin \mathrm D(\eta_2,\eta_3)$, by denoting   $x^{(\eta_2,\eta_3)}=(x_2,x_3)=(+,-)$ we can, as above, choose and integer $h$ such that $\eta_2+h\eta_3<x<\eta_2+(h+1)\eta_3$. Then, taking $d=\eta_2+h\eta_3$ and arguing as in (a) with $d=\eta_2+h\eta_3$ and $x=x_2\eta_2+x_3\eta_3$,  we have that  $0<1<x_2$ and $x_3<h<0$. Now, as  $\eta_3=\eta_1\bmod \eta_2=\eta_1-\peb{{\eta_1}/{\eta_2}}\eta_2$, we can rewrite
        	\begin{align*}
        	x & =  x_2\eta_2+x_3(\eta_1-\peb{{\eta_1}/{\eta_2}}\eta_2) = x_3\eta_1+(x_2-x_3\peb{{\eta_1}/{\eta_2}})\eta_2, \\
        	d & =  \eta_2+h(\eta_1-\peb{{\eta_1}/{\eta_2}}\eta_2) = h\eta_1+(1-h\peb{{\eta_1}/{\eta_2}})\eta_2. 
        	\end{align*}
        	Thus, $(x_1',x_2')^{(\eta_1,\eta_2)}=(x_3,x_2-x_3\peb{{\eta_1}/{\eta_2}})$ and $(d_1',d_2')^{(\eta_1,\eta_2)}=(h,1-h\peb {{\eta_1}/{\eta_2}})$. Moreover, it is clear that $x_1'=x_3<h=d_1'<0$.
        	We also have that $-x_3\peb {{\eta_1}/{\eta_2}} > -h\peb{{\eta_1}/{\eta_2}}>0$ and $x_2>1>0$. Hence  $x_2'=x_2-x_3\peb {{\eta_1}/{\eta_2}}>1-h\peb {{\eta_1}/{\eta_2}}=d_2'>0$. 
     \end{enumerate}
    \end{enumerate}    
        \item For the general case, we will follow the arguments of the preceding cases. It is clear that for all suitable $j$,  $\gcd(\eta_j,\eta_{j+1})= \gcd(\delta_1,\delta_2)=g$. Hence we have the following.
          \begin{enumerate}[$\bullet$]
          	\item Observe that $j\in I$ implies $\eta_{j+1}>0$, so an element $d\in\{\eta_{j+1},\eta_{j+1}+g,\ldots ,\eta_{j}\}$ is in $\mathrm D(\eta_{j-1},\eta_j)\setminus \{\eta_j\}$ if and only if $d^{(j-1)}_1=1$, where   $d^{(\eta_{j-1},\eta_j)}= (d^{(j-1)}_1,d^{(j-1)}_2)$. This is the same argument used in Proposition \ref{premain}, applied for $\eta_1=\eta_{j-1}$ and for $\eta_2=\eta_j$.  
          	\item Working at this $(j-1,j)$ level, it is also clear that, if $x\in \{\eta_{j+1},\eta_{j+1}+g,\ldots, \eta_{j}\}\setminus (\mathrm D(\eta_{j-1},\eta_j)\cup \mathrm D(\eta_j,\eta_{j+1}))$, then 
          	\begin{enumerate}[(a)]
          		\item if $x^{(\eta_{j-1},\eta_j)}=(x^{(j-1)}_1,x^{(j-1)}_2)=(+,-)$, then there exists $d<x$, $d\in \mathrm D(\eta_{j-1},\eta_{j})$, such that $d^{(\eta_{j-1},\eta_j)}=(d^{(j-1)}_1,d^{(j-1)}_2)$ with $0<d^{(j-1)}_1<x^{(j-1)}_1$ and $x^{(j-1)}_2<d^{(j-1)}_2<0$; and
          		\item if $x'^{(\eta_{j-1},\eta_j)}=(x'^{(j-1)}_1,x'^{(j-1)}_2)=(-,+)$, then there exists $d<x$, $d\in \mathrm D(\eta_{j},\eta_{j+1})$, such that $d'^{(\eta_{j-1},\eta_j)}=(d'^{(j-1)}_1,d'^{(j-1)}_2)$ with $x'^{(j-1)}_1\le d'^{(j-1)}_1\le 0$ and $0<d'^{(j-1)}_2<d'^{(j-1)}_2$.
          	\end{enumerate}
          	We only need to apply the preceding cases at this $(j-1,j)$ level.
          \end{enumerate}
       	It remains to prove that the inequalities for the $(j-1,j)$ level (denoted by the bracketed superscript) can be translated to the first level. To this end, we proceed as in the previous point.
       	  \begin{enumerate}
       	  	\item Assume that $0\le d^{(j-1)}_1<x^{(j-1)}_1$ and $x^{(j-1)}_2<d^{(j-1)}_2<0$. Notice that  $d^{(j-1)}_1=0$ if and only if $d=\eta_{j}$. As $x=x^{(j-1)}_1\eta_{j-1}+x^{(j-1)}_2\eta_j$, $d=d^{(j-1)}_1\eta_{j-1}+d^{(j-1)}_2\eta_j$ and $\eta_j=\eta_{j-2}-\peb {\frac{\eta_{j-2}}{\eta_{j-1}}}\eta_{j-1}$, we have that 
       	      \[
       	       	\begin{aligned}
       	       	x & =  x^{(j-1)}_1\eta_{j-1}+x^{(j-1)}_2\left(\eta_{j-2}-\peb{\tfrac{\eta_{j-2}}{\eta_{j-1}}} \eta_{j-1}\right ) \\
       	       	 & =  x^{(j-1)}_2\eta_{j-2}+\left(x^{(j-1)}_1-x^{(j-1)}_2\peb{\tfrac{\eta_{j-2}}{\eta_{j-1}}}\right )\eta_{j-1}, \\
       	       	d & =   d^{(j-1)}_1\eta_{j-1}+d^{(j-1)}_2\left(\eta_{j-2}-\peb{\tfrac{\eta_{j-2}}{\eta_{j-1}}}\eta_{j-1}\right ) \\
       	       	& =  d^{(j-1)}_2\eta_{j-2}+\left(d^{(j-1)}_1-d^{(j-1)}_2\peb{\tfrac{\eta_{j-2}}{\eta_{j-1}}}\right )\eta_{j-1}. 
       	       	\end{aligned}
       	      \]
       	      From these equalities we can deduce that $x'^{(j-2)}_1=x^{(j-1)}_2<d^{(j-1)}_2=d'^{(j-2)}_1<0$ and  $x'^{(j-2)}_2= x^{(j-1)}_1-x^{(j-1)}_2\peb{\tfrac{\eta_{j-2}}{\eta_{j-1}}}> d^{(j-1)}_1-d^{(j-1)}_2\peb{\tfrac{\eta_{j-2}}{\eta_{j-1}}}= d'^{(j-2)}_2>0$, where $x'^{(\eta_{j-2},\eta_{j-1})}=(x'^{(j-2)}_1,x'^{(j-2)}_2)=(-,+)$, and $d'^{(\eta_{j-2},\eta_{j-1})}=(d'^{(j-2)}_1,d'^{(j-2)}_2)=(-,+)$. Notice that we can ensure that in this setting all the inequalities are strict.
        	     \item Assume that  $0\ge d'^{(j-1)}_1\ge x'^{(j-1)}_1$ and $x'^{(j-1)}_2>d'^{(j-1)}_2>0$. Since $x=x'^{(j-1)}_1\eta_{j-1}+x'^{(j-1)}_2\eta_j$, $d=d'^{(j-1)}_1\eta_{j-1}+d'^{(j-1)}_2\eta_j$ and $\eta_j=\eta_{j-2}-\peb{\frac{\eta_{j-2}}{\eta_{j-1}}}\eta_{j-1}$ we have 
       	     \[
       	     \begin{aligned}
       	     x & =  x'^{(j-1)}_1\eta_{j-1}+x'^{(j-1)}_2\left(\eta_{j-2}-\peb{\tfrac{\eta_{j-2}}{\eta_{j-1}}}\eta_{j-1}\right ) \\
       	        & = x'^{(j-1)}_2\eta_{j-2}+\left(x'^{(j-1)}_1-x'^{(j-1)}_2\peb{\tfrac{\eta_{j-2}}{\eta_{j-1}}}\right )\eta_{j-1}, \\
       	     d  & =   d'^{(j-1)}_1\eta_{j-1}+d'^{(j-1)}_2\left(\eta_{j-2}-\peb{\tfrac{\eta_{j-2}}{\eta_{j-1}}}\eta_{j-1}\right ) \\
       	     & = d'^{(j-1)}_2\eta_{j-2}+\left (d'^{(j-1)}_1-d'^{(j-1)}_2\peb{\tfrac{\eta_{j-2}}{\eta_{j-1}}}\right )\eta_{j-1}. 
       	     \end{aligned}
       	     \]
       	     Again, it follows that  $x^{(j-2)}_1=x'^{(j-1)}_2>d'^{(j-1)}_2=d^{(j-2)}_1>0$ and $x^{(j-2)}_2= x'^{(j-1)}_1-x'^{(j-1)}_2\peb{\tfrac{\eta_{j-2}}{\eta_{j-1}}}<d'^{(j-1)}_1-d'^{(j-1)}_2\peb{\tfrac{\eta_{j-2}}{\eta_{j-1}}}=d^{(j-2)}_2\le 0$.
       	  \end{enumerate}
Notice that the inequality is strict because $\peb{\frac{\eta_{j-2}}{\eta_{j-1}}}\ge 1$ and we have $x'^{(j-1)}_2<d'^{(j-1)}_2$. Also $d^{(j-1)}_1=0$ implies $d=d'^{(j-1)}_2\eta_{j-2}$, but since $d\in \mathrm {Euc}(\delta_1,\delta_2)$, we deduce $d'^{(j-1)}_2=1$, that is, we always can take $d=\eta_{j+1}$ in the $(-,+)$ case to start the process.
       	  
In this way, we can apply alternatively the preceding processes to obtain the desired result. Notice that only if we start with (b), we can have $d'^{j-1}_1=0$. Hence, only when $j-1=1$ (if we perform another step, when applying (a) the inequality becomes strict), we will have that the $(\eta_1,\eta_2)$-coordinates of $d$ would be  $(0,1)$, that is, $d=\eta_2$.
       	  
However, as the statement of (3) assumes $j>2$, the inequalities will always be strict.\qedhere
 \end{enumerate}
\end{proof}

We will refer to $d$ defined in Proposition \ref{main} as the \emph{basement} of $x=(x_1,x_2)$, denoted by $\mathrm{bsm}(x)$. Notice that the basement of $x$ may not be unique: however, for our purposes, we will consider as the basement of $x$ any $d$ satisfying the properties stated in Proposition \ref{main}. 

\begin{example}\label{mainexample-pre2}
In the setting of Example \ref{mainexample-pre}, consider $x=35$. Since $33=\eta_4 \le x \le 109 = \eta_3$, according to Proposition \ref{main} we can take $j=3$.
As $35=-13\eta_3+44\eta_4=20\eta_3-65\eta_4$; we have $x^{(\eta_3,\eta_4)}=(20,-65)$ and $x'^{(\eta_3,\eta_4)}=(-13,44)$. Remember that these decompositions are unique using Proposition \ref{unique}.

For $x'^{(\eta_3,\eta_4)}=(-13,44)$ we know that $d'=\eta_4=33$ and  $d'^{(\eta_3,\eta_4)}=(0,1)$; while we look for $d\in\mathrm{D}(\eta_3,\eta_4)=\{76,43,10\}$. As $10<x=35<43$ we take $d=10$ and then $d^{(\eta_3,\eta_4)}=(1,-3)$. If we translate them to level $(\delta_1,\delta_2)$, we have: $35=-57\delta_1+158\delta_2=85\delta_1-235\delta_2$, and $(x_1,x_2)=(85,-235)$ and $(x_1'x_2')=(-57,158)$. Thus, considering $d'=\eta_4=33$, as $33=-\eta_1+3\eta_2$, we have $d'^{(\delta_1,\delta_2)}=(d'_1,d'_2)=(-1,3)$. For the other case, by considering $d=10$,  we obtain $10=4\eta_1-11\eta_2$, that is, $d^{(\delta_1,\delta_2)}= (d_1,d_2)=(4,-11)$. So we have $x'_1<d'_1<0<d'_2<x'_2$ and $x_2<d_2<0<d_1<x_1$.

Observe that $33$ is the $d'$ that yields Proposition \ref{main}, but it may happen that there are other elements in $\mathrm{Euc}(\delta_1,\delta_2)$ satisfying our purposes too.

It is not difficult to see that for any element in $\{34,...,42\}$ we obtain the same results, that is, $d=10$ and $d'=33$, while if we take a number comprised between $44$ and $75$, we will obtain $d=43$ and $d'=33$.

Let us, now, consider $x=15$ which is between $\eta_5=10$ and $\eta_4=33$. In this case $j=4$, then we can write  $15=-5\eta_4+18\eta_5=5\eta_4-15\eta_5$. So, for $x'^{(\eta_4,\eta_5)}=(-5,18)$ we know that $d'=\eta_5=10$ and  $d'^{(\eta_4,\eta_5)}=(0,1)$; while as $d\in\mathrm{D}(\eta_4,\eta_5)=\{23,13,3\}$, we can take $d=13<x=15$ which yields $d^{(\eta_4,\eta_5)}=(1,-2)$, satisfying Proposition \ref{main} as  $x^{(\eta_4,\eta_5)}=(5,-15)$. Observe that, in this case, when we compute the basements at the first level, we obtain $d=10$ because $d^{(\delta_1,\delta_2)}=(4,-11)$, while $x^{(\delta_1,\delta_2)}=(77,-213)$, and $d'=33$ with $d'^{(\delta_1,\delta_2)}=(d'_1,d'_2)=(-1,3)$ for $x'^{(\delta_1,\delta_2)}=(-65,180)$, changing the roles of $d$ and $d'$ from level $j=4$ to the first level.



\end{example}

Our next goal is to associate $\mathrm{bsm}(x)$ to a vector in $M_S$. We will have two such vectors per case, associated to the decompositions in Proposition \ref{unique}.

As above $g=\gcd(\delta_1,\delta_2)$. For  $x\in \{g,2g,\ldots,\max\{\delta_1,\delta_2\}\}$, and $\{\mathbf v_1, \mathbf v_2\}$ the basis for $M_S$ chosen in Section \ref{sec:uno}, set 
\begin{equation}\label{w-x}
\mathbf w_x=x_1\mathbf v_1+\sigma x_2\mathbf v_2, \quad \mathbf w'_x=x'_1\mathbf v_1+ \sigma x'_2\mathbf v_2,
\end{equation}
with $(x_1,x_2)=x^{(\delta_1,\delta_2)}$ and $(x_1',x_2')=x'^{(\delta_1,\delta_2)}$.
Notice that in this setting  $\ell(\mathbf w_x)=x_1\delta_1+x_2\delta_2= x= x_1'\delta_1+x_2'\delta_2=\ell(\mathbf w_{x}')$.

In the following, we will use the bracketed subscript to refer to the coordinates of vectors. For instance, we would have, for a generic vector $\mathbf v$, $\mathbf v=(v_{(1)},v_{(2)},v_{(3)})$. In particular,  
we will use the following notation to refer to the coordinates of the vectors $\mathbf w_x$ and $\mathbf w_x'$ we just introduced: 
\[\mathbf w_x=(w_{x_{(1)}},w_{x_{(2)}},w_{x_{(3)}}) \hbox{ and }\mathbf w_x'=(w_{x_{(1)}}',w_{x_{(2)}}',w_{x_{(3)}}').\]
As $x_1>0$ and $x_2\le 0$, following an argument similar to the proof of Proposition \ref{signos-v-deltas}, it can be shown that the signs of the coordinates of $\mathbf w_x$ and $\mathbf w_x'$ are as in Table \ref{signs-vx}.
\begin{table}[h]
\begin{center}
\begin{tabular}{c|cc}
$\sigma$   & $\mathbf w_x$ & $\mathbf w'_x$ \\
\hline
1 & $(?,-,+)$ & $(?,+,-)$ \\
-1  & $(+,?,-)$ & $(-,?,+)$ 
\end{tabular}
\end{center}
\caption{Signs for $\mathbf w_x$ and $\mathbf w_x'$}
\label{signs-vx}
\end{table}

Remember that $\mathbf v_1=(+,-,0)$ and $\mathbf v_2=(+,+,-)$.

The following corollary shows that, for $x$ as in Proposition \ref{main}, the vector $\mathrm{bsm}(x)$ has two coordinates that are in absolute value lower than those of $\mathbf w_x$ or $\mathbf w_x'$.

\begin{corollary}\label{elmenor}
Let the notations and hypotheses be as above. Let $\mathbf v=a_1 \mathbf v_1+a_2\sigma\mathbf v_2$, with $a_1$ and $a_2$ integers such that $a_1a_2\le 0$, $-\delta_2/g<a_1\le \delta_2/g$ and  $-\delta_1/g<a_2\le \delta_1/g$. Set $x=\ell(\mathbf v)\in \{g,2g,\ldots, \max\{\delta_1,\delta_2\}\}\setminus \mathrm {Euc}(\delta_1,\delta_2)$, and let $d=\mathrm{bsm}(x)$. 
		\begin{enumerate}[(i)]
			\item If $a_1>0$ (that is, $\mathbf v=\mathbf w_x$), there exists $i\in\{1,2\}$ such that $|w_{d_{(i)}}|<|w_{x_{(i)}}|$  and $|w_{d_{(3)}}|<|w_{x_{(3)}}|$. Moreover, for such an $i$ we have that  $w_{d_{(i)}}w_{d_{(3)}}<0$.
			\item If $a_1<0$ (that is, $\mathbf v=\mathbf w_x'$), there exists $i\in\{1,2\}$ such that $|w'_{d_{(i)}}|<|w'_{x_{(i)}}|$ and $|w'_{d_{(3)}}|<|w'_{x_{(3)}}|$. Moreover, for such an $i$ we have that $w'_{d_{(i)}}w'_{d_{(3)}}<0$. 
		\end{enumerate}
\end{corollary}
\begin{proof}

Notice that $x=a_1\delta_1+a_2\delta_2$. According to Proposition \ref{unique}, if $a_1>0$, then $(a_1,a_2)=x^{(\delta_1,\delta_2)}$;  otherwise $(a_1,a_2)=x'^{(\delta_1,\delta_2)}$. Hence, in the first case $\mathbf v=\mathbf w_x$, while in the second case $\mathbf v=\mathbf w'_x$.

Suppose that $a_1>0$. Following Proposition \ref{main}, consider $d=d_1\delta_1+d_2\delta_2$ with $a_1>d_1> 0$ and $0>d_2>a_2$. 

\begin{itemize}
\item If $\sigma=1$, $w_{x_{(3)}}=a_1 v_{1_{(3)}}+a_2v_{2_{(3)}}=a_2v_{2_{(3)}}>d_2v_{2_{(3)}}=d_1v_{1_{(3)}}+d_2v_{2_{(3)}}=w_{d_{(3)}}>0$. We take $i=2$, obtaining $w_{x_{(2)}}=a_1 v_{1_{(2)}}+a_2v_{2_{(2)}}<d_1v_{1_{(2)}}+d_2v_{2_{(2)}}<0$.

\item If $\sigma=-1$, $w_{x_{(3)}}=a_1 v_{1_{(3)}}-a_2v_{2_{(3)}}=-a_2v_{2_{(3)}}<-d_2v_{2_{(3)}}=d_1v_{1_{(3)}}-d_2v_{2_{(3)}}=w_{d_{(3)}}<0$. Also  $w_{x_{(1)}}=a_1 v_{1_{(1)}}-a_2v_{2_{(1)}}>d_1v_{1_{(1)}}-d_2v_{2_{(1)}}=w_{d_{(1)}}>0$.
\end{itemize}

Now assume that $a_1<0$. Following Proposition \ref{main} consider $d=d'_1\delta_1+d'_2\delta_2$ with $a_1<d'_1\le 0$ and $0<d'_2<a_2$.

\begin{itemize}
\item If $\sigma=1$, $0>w'_{d_{(3)}}= d_1'v_{1_{(3)}}+d_2'v_{2_{(3)}}=d_2'v_{2_{(3)}}>a_1v_{1_{(3)}}+a_2v_{2_{(3)}}=w'_{x_{(3)}}$. Now take $i=2$. Then $0<w'_{d_{(2)}}= d_1'v_{1_{(2)}}+d_2'v_{2_{(2)}}<a_1v_{1_{(2)}}+a_2v_{2_{(2)}}=w'_{x_{(2)}}$. 

\item If $\sigma=-1$, $0<w'_{d_{(3)}}= d_1'v_{1_{(3)}}-d_2'v_{2_{(3)}}=-d_2'v_{2_{(3)}}<-a_2v_{2_{(3)}}=a_1v_{1_{(3)}}-a_2v_{2_{(3)}}=w'_{x_{(3)}}$. Here we choose $i=1$, obtaining $0>w'_{d_{(1)}}=d_1'v_{1_{(1)}}-d_2'v_{2_{(1)}}>a_1v_{1_{(1)}}-a_2v_{2_{(1)}}=w'_{x_{(1)}}$.\qedhere
\end{itemize}

\end{proof}

This corollary ensures that for every $\mathbf w_x$ or $\mathbf w'_x$ there exists a  $\mathbf w_d$ or $\mathbf w'_d$, respectively, with $d<x$ and such that the two known coordinates of  $\mathbf w_x$ or $\mathbf w_x'$, according to Table \ref{signs-vx}, are greater in absolute value than those of $\mathbf w_d$ or $\mathbf w_d'$, respectively.

Now we are going to show how can uniqueness in Proposition \ref{unique} extend to the level of the vectors $\mathbf w_x$ and $\mathbf w'_x$. We want to associate to each $\mathbf v \in M_S$  such that $\ell(\mathbf v)\in\{0,g,2g,\ldots, \max\{\delta_1,\delta_2\}\}$ two vectors $\mathbf w_x$, $\mathbf w'_x$ such that $\ell(\mathbf w_x)=\ell(\mathbf w'_x)=\ell(\mathbf v)$.

Next we see that if a vector has length equal to zero, then it is a multiple of the vector appearing in Proposition \ref{signos-v-deltas}. Hence, when this vector is added or subtracted to another vector, the length remains the same.

\begin{lemma}
With the notation defined in Section \ref{sec:uno} and Definition \ref{Euclidset}, let $\mathbf v\in M_S$ with $\ell(\mathbf v)=0$. Then 
\[
\mathbf v=\alpha(\delta_2/g\mathbf v_1-\sigma\delta_1/g\mathbf v_2),
\]
for some $\alpha\in \mathbb Z$. 
\end{lemma}
\begin{proof}
Since $\{\mathbf v_1, \sigma\mathbf v_2\}$ is a basis for $M_S$, we have that there exists $\lambda_1,\lambda_2\in \mathbb Z$ such that $\mathbf v=\lambda_1\mathbf v_1+\lambda_2\sigma\mathbf v_2$. Now $0=\ell(\mathbf v)=\lambda_1\delta_1+\lambda_2\delta_2$. Since $\gcd(\delta_1,\delta_2)=g$, we deduce that $\delta_1/g$ divides $\lambda_2$, and $\delta_2/g$ divides $\lambda_1$. So there exists $k_1$ and $k_2$ integers such that $\lambda_1=k_1\delta_2/g$ and $\lambda_2=k_2\delta_1/g$. We then have $0=k_1+k_2$ and consequently $k_1=-k_2$. Take $\alpha=k_1/g$. Hence $\mathbf v=\alpha\delta_2 \mathbf v_1 - \alpha \delta_1 \sigma\mathbf v_2$. 
\end{proof}

We now characterize the vectors in $M_S$ with length in $\{1,\ldots, \max\{\delta_1,\delta_2\}\}$.

\begin{proposition}\label{normalizacion} 
Under the hypotheses and notations of Section \ref{sec:uno} and Definition \ref{Euclidset}, let $\mathbf v\in M_S$ with $0< \ell(\mathbf v) \le \max\{\delta_1,\delta_2\}$. Then
\[
\mathbf v=a_1\mathbf v_1+\sigma a_2\mathbf v_2+\alpha(\delta_2/g\mathbf v_1-\sigma\delta_1/g\mathbf v_2),
\]
for some $a_1,a_2\in \mathbb Z^2\setminus\{(0,0)\}$ such that $a_1a_2\le 0$, $-\delta_2/g<a_1\le 
\delta_2/g$, $-\delta_1/g<a_2\le \delta_1/g$, where $g=\gcd(\delta_1,\delta_2)$ and $\alpha\in\mathbb Z$.

Furthermore, $a_1$, $a_2$ and $\alpha$ are unique if we impose that the third coordinates of $\mathbf u=a_1\mathbf v_1+\sigma a_2\mathbf v_2$ and $\mathbf v$ have the same sign.
\end{proposition}
\begin{proof} 
We use once more that $\{\mathbf v_1,\sigma\mathbf v_2\}$ is a generating system for $M_S$; then there exist unique $\lambda_1,\lambda_2\in \mathbb Z$ such that $\mathbf v=\lambda_1 \mathbf v_1 +\lambda_2\sigma\mathbf v_2$. 

Let us prove that $\lambda_1\lambda_2\le 0$. Clearly, if $\lambda_1\lambda_2=0$, we are done. 

Assume that $\lambda_1>0$ (the other case is analogous). 
\begin{itemize}
 \item If $\delta_1>\delta_2$, we have $0<\ell(\mathbf v)=\lambda_1\delta_1+\lambda_2\delta_2\le \max\{\delta_1,\delta_2\}=\delta_1$. Then $\lambda_2\delta_2\le (1-\lambda_1)\delta_1\le 0$, which implies  $\lambda_2\le 0$.
 \item If $\delta_2>\delta_1$, we have $0<\ell(\mathbf v)=\lambda_1\delta_1+\lambda_2\delta_2\le \max\{\delta_1,\delta_2\}=\delta_2$. Then $0<\lambda_1\delta_1\le (1-\lambda_2)\delta_2$ and $0<1-\lambda_2$, yielding $\lambda_2\le 0$.
\end{itemize} 

Next, we prove the following assertions. 

\begin{table}[h]
\begin{center}
\begin{tabular}{lcl}
$\lambda_1>\delta_2/g$ & implies & $\lambda_2\le -\delta_1/g$, \\
$\lambda_1\le -\delta_2/g$ & implies & $\lambda_2>\delta_1/g$, \\
$-\delta_2/g<\lambda_1\le 0$ & implies & $0<\lambda_2\le \delta_1/g$, \\
$0<\lambda_1\le \delta_2/g$ & implies & $-\delta_1/g<\lambda_2\le 0$. \\
\end{tabular}
\end{center}
\end{table}

\begin{enumerate}[(1)]
\item Assume that $\lambda_1>\delta_2/g$.
 \begin{itemize}
 \item If $\delta_1>\delta_2$, we have that  $0<\ell(\mathbf v)=\lambda_1\delta_1+\lambda_2\delta_2\le \delta_1$.  Hence $\lambda_2\delta_2\le (1-\lambda_1)\delta_1$. From $\lambda_1>\delta_2/g$, we deduce  $(1-\lambda_1)\le-\delta_2/g$.  Therefore, $\lambda_2\delta_2\le (1-\lambda_1)\delta_1\le-\delta_1\delta_2/g$, and thus $\lambda_2\le -\delta_1/g$.
 \item If $\delta_2>\delta_1$, we have $0<\ell(\mathbf v)=\lambda_1\delta_1+\lambda_2\delta_2\le \delta_2$, and then $\lambda_1\delta_1\le (1-\lambda_2)\delta_2$. As $\delta_2/g<\lambda_1$, we get  $\delta_2\delta_1/g<\lambda_1\delta_1$. This implies  $\delta_1/g<(1-\lambda_2)$ or equivalently $\lambda_2\le -\delta_1/g$.
 \end{itemize}
\item Assume now that $\lambda_1\le -\delta_2/g$. Then  $0<\ell(\mathbf v)=\lambda_1\delta_1+\lambda_2\delta_2\le -\delta_2\delta_1/g+\lambda_2\delta_2=(\lambda_2-\delta_1/g)\delta_2$. This forces $\lambda_2>\delta_1/g$.
\item Suppose that $-\delta_2/g<\lambda_1\le 0$.
 \begin{itemize}
 \item If $\delta_1>\delta_2$, then $0<\ell(\mathbf v)=\lambda_1\delta_1+\lambda_2\delta_2\le \delta_1$. As $(1-\lambda_1)\le\delta_2/g$, we deduce that  $\lambda_2\delta_2\le (1-\lambda_1)\delta_1\le\delta_1\delta_2/g$, and consequently $\lambda_2\le \delta_1/g$.
 \item If $\delta_2>\delta_1$, we have $\ell(\mathbf v)=\lambda_1\delta_1+\lambda_2\delta_2\le \delta_2$. Since $-\delta_2/g<\lambda_1$, we obtain  $-\delta_2\delta_1/g+\lambda_2\delta_2<\lambda_1\delta_1+\lambda_2\delta_2\le \delta_2$. Hence  $\lambda_2-\delta_1/g<1$, or equivalently, $\lambda_2\le \delta_1/g$.
 \end{itemize}
\item Finally, assume that $0<\lambda_1\le \delta_2/g$. Then $0<\ell(\mathbf v)=\lambda_1\delta_1+\lambda_2\delta_2\le \delta_2\delta_1/g+\lambda_2\delta_2=(\lambda_2+\delta_1/g)\delta_2$. Hence $-\delta_1/g<\lambda_2$.
\end{enumerate}
In order to conclude the proof it suffices to observe the following. 
\begin{itemize}
	\item If $\lambda_1$ fulfills the conditions (3) or (4), we are done.
  	\item If $\lambda_1>\delta_2/g$, take $\alpha\in \mathbb Z$ such that $0<\lambda_1-\alpha\delta_2/g\le \delta_2/g$. Set $a_1=\lambda_1-\alpha\delta_2/g\le \delta_2/g$ and $a_2=\lambda_2+\alpha\delta_1/g$. Define $\mathbf u=a_1\mathbf v_1+\sigma a_2\mathbf v_2$. Then $\ell(\mathbf v)=\ell(\mathbf u)$, and $\mathbf v=\mathbf u+\alpha(\delta_2/g\mathbf v_1-\sigma\delta_1/g\mathbf v_2)$. Now $\mathbf u$ fulfills (4), and we are done. Notice that if we take $\alpha+1$ instead of $\alpha$, we fall in case (3), with $\mathbf u'=\mathbf u-(\delta_2/g\mathbf v_1-\sigma\delta_1/g\mathbf v_2)$: thus it can be deduced that there are two possible choices of $\alpha$ fulfilling the conditions of the statement.
  	\item If $\lambda_1<-\delta_2/g$, we proceed analogously.
\end{itemize}
Since there are two possible choices of $\alpha$, but $\lambda_1$ and $\lambda_2$ are unique, there exist two possible choices for $(a_1,a_2)$. In order to conclude the proof, it suffices to show that the signs of the third coordinates of $\mathbf u$ and $\mathbf u'$ are different (and thus one of them must be equal to the sign of the third coordinate of $\mathbf v$).  Assume that $\alpha$ is such that $0<a_1\le \delta_2/g$, and thus by (4) applied to $\mathbf u$, $-\delta_1/g<\sigma a_2\le 0$. 

If we write $\mathbf u'=a_1'\mathbf v_1+\sigma a_2'\mathbf v_2$ then, as $a_1'=a_1-\delta_2/g>0$, by (3) we have $\sigma a_2'>0$. Since $a_2$ and $a_2'$ have different signs, we are done.
\end{proof}

This proposition will be used later to work with $\mathbf u$ instead of $\mathbf v$. 


The goal of this work is to prove the following theorem. For clarity's sake, we first give an example; then, we provide the intermediate results needed for its proof, which will be a direct consequence of Lemmas \ref{noesta} and \ref{siesta}.

\begin{theorem}\label{unionDs}
Let $S$ be a symmetric numerical semigroup with embedding dimension three. With the notation introduced in Section \ref{sec:uno} and Definition \ref{Euclidset}, we have 
\[
\Delta(S)=\mathrm{Euc}(\delta_1,\delta_2)\setminus\{0\}.
\]
\end{theorem}

\begin{example}\label{mainexample}
Let $S=\langle s_2s_3,s_1s_3,s_1s_2\rangle=\langle 2015, 7124, 84940\rangle$, with $s_1=548$, $s_2=155$, and $s_3=13$. Then $\mathbf v_1=(548,-155,0)$ and $\mathbf v_2=(0,155,-13)$. Hence
$\delta_1=393$ and $\delta_2=142$. Extended Euclid's Algorithm for $\delta_1=393$ and $\delta_2=142$ yields the Table \ref{tabla-ejemplo}, whose cells contain the vectors $\mathbf w_d$ or $\mathbf w_d'$, as defined in \eqref{w-x}, for $d\in \mathrm{Euc}(\delta_1,\delta_2)$.

\begin{table}[h]\label{tabla-ejemplo}
{\footnotesize
\begin{tabular}{l|ccc}
 \hline
  $\mathbf w_d = (?,-,+)$ & $\begin{aligned} \mathbf w_{393} & =\mathbf v_1 \\ &=(548, -155, 0) \end{aligned}$ & $\begin{aligned}\mathbf w_{251}& =\mathbf v_1-\mathbf v_2\\ &= (548, -310, 13)\end{aligned}$ & $\begin{aligned} \mathbf w_{109} &=\mathbf v_1-2\mathbf v_2\\&= (548, -465, 26)\end{aligned}$ \\ 
  $\mathrm D(\delta_1,\delta_2)$ & $\delta_1=393$ & $\delta_1-\delta_2=251$ & $\delta_1-2\delta_2=109$  \\

 \hline
   $\mathbf w_d'=(?,+,-)$ & $\begin{aligned} \mathbf w_{142}' & = \mathbf v_2 \\ & = (0, 155, -13)\end{aligned}$ & $\begin{aligned} \mathbf w_{33}' & = -\mathbf v_1+3\mathbf v_2 \\ &= (-548, 620, -39) \end{aligned}$  & \\ 
  $\mathrm D(\delta_2, \delta_3)$ &$\delta_2=142$ & $\delta_2-\delta_3=33$  & \\ 
  
 \hline
  $\mathbf w_d = (?,-,+)$ & $\begin{aligned} \mathbf w_{76} & =\mathbf 2v_1-5\mathbf v_2 \\ &=(1096, -1085, 65) \end{aligned}$ & $\begin{aligned}\mathbf w_{43}& =\mathbf 3\mathbf v_1-8\mathbf v_2\\ &= (1644, -1705, 104)\end{aligned}$ & $\begin{aligned} \mathbf w_{10} &=4\mathbf v_1-11\mathbf v_2\\&= (2192, -2325, 143)\end{aligned}$\\ 
  $\mathrm D(\delta_3,\delta_4)$ & $\delta_3-\delta_4=76$ & $\delta_3-2\delta_4=43$ & $\delta_3-3\delta_4=10$  \\

 \hline

   $\mathbf w_d'=(?,+,-)$ &$\begin{aligned} \mathbf w_{23}' &=-5\mathbf v_1+14\mathbf v_2 \\ &= ( -2740, 2945, -182 )\end{aligned}$ & $\begin{aligned} \mathbf w_{13}'& =-9\mathbf v_1+25\mathbf v_2 \\&= ( -4932, 5270, -325)\end{aligned}$ & $\begin{aligned} \mathbf w_{3}' & = -13\mathbf v_1+36\mathbf v_2 \\& (-7124, 7595, -468)\end{aligned}$ \\
 $\mathrm D(\delta_4,\delta_5)$  &  $\delta_4-\delta_5=23$ & $\delta_4-2\delta_5=13$ & $\delta_4-3\delta_5=3$ \\
 
\hline
  $\mathbf w_d = (?,-,+)$ & $\begin{aligned} \mathbf w_{7} & =17\mathbf v_1-47\mathbf v_2\\ &=(9316, -9920, 611) \end{aligned}$ & $\begin{aligned}\mathbf w_{4}& =30\mathbf v_1-83\mathbf v_2\\ &= (16440, -17515, 1079)\end{aligned}$ & $\begin{aligned} \mathbf w_{1} &=43\mathbf v_1-119\mathbf v_2\\&= (23564, -25110, 1547)\end{aligned}$ \\ 
  $\mathrm D(\delta_5,\delta_6)$ & $\delta_5-\delta_6=7$ & $\delta_5-2\delta_6=4$ & $\delta_5-3\delta_6=1$ \\
  
\hline
 $\mathbf w_d'=(?,+,-)$ & $\begin{aligned} \mathbf w_{2} & =-56\mathbf v_1+155\mathbf v_2 \\ &=(-30688, 32705, -2015) \end{aligned}$ &  $\begin{aligned} \mathbf w_1 & =-99\mathbf v_1+274\mathbf v_2 \\ &=(-54252, 57815, -3562)\end{aligned}$ & $\begin{aligned} \mathbf w_0 & =-142\mathbf v_1+393\mathbf v_2\\  &=( -77816, 82925, -5109) \end{aligned}$  \\ 
 $\mathrm D(\delta_6,\delta_7)$ & $\delta_6-\delta_7=2$ & $\delta_6-2\delta_7=1$ & $\delta_6-3\delta_7=0$ \\ 
 
\hline
\end{tabular}
}\caption{Table for $\langle 2015, 7124, 84940\rangle$}
\end{table}

According to Theorem \ref{unionDs},  we would obtain
\[
\Delta(S)=\{1,2,3,4,7,10,13,23,33,43,76,109,142,251,393\}.
\]
\end{example}

\begin{lemma}
Let $\mathbf z,\mathbf z'\in\mathsf Z(s)$, with $0\neq s\in S$. With the notation introduced in Section \ref{sec:uno}, if $0<\ell(\mathbf z)-\ell(\mathbf z')<\max\{\delta_1,\delta_2\} $, then
 \begin{enumerate}
 \item $\mathbf z-\mathbf z'$ is either of the form $(?,+,-)$ or $(?,-,+)$ when $\ell(\mathbf v_2)>0$ ($\sigma=1$), and
 \item $\mathbf z-\mathbf z'$ is of the form $(+,?,-)$ or $(-,?,+)$ when $\ell(\mathbf v_2)<0$ ($\sigma=-1$).
 \end{enumerate}
 \end{lemma}
\begin{proof} 
Recall that  $\mathbf z,\mathbf z'\in\mathbb N^3$. Since $\mathbf z-\mathbf z'\in M_S$ and $0<\ell(\mathbf z-\mathbf z')\le \max\{\delta_1,\delta_2\}$ by Proposition \ref{normalizacion} we know that $\mathbf z-\mathbf z'=a_1\mathbf v_1+\sigma a_2\mathbf v_2+\alpha(\delta_2/g\mathbf v_1-\sigma\delta_1/g\mathbf v_2)$. If $\alpha>0$, we have that $\mathbf z-\mathbf z'=(\alpha\delta_2/g+a_1)\mathbf v_1+\sigma(a_2-\alpha\delta_1/g)\mathbf v_2$, and clearly $(\alpha\delta_2/g+a_1)>0$ and $(a_2-\alpha\delta_1)/g<0$. According to Table \ref{signs-vx}, $\mathbf z-\mathbf z'=(?,-,+)$ if $\sigma=1$, or $\mathbf z-\mathbf z'=(+,?,-)$ if $\sigma=-1$. A similar argument shows that if $\alpha<0$, then $\mathbf z-\mathbf z'=(?,+,-)$ for $\sigma=1$, and $\mathbf z-\mathbf z'=(-,?,+)$ for $\sigma=-1$.

For the case $\alpha=0$, as $a_1a_2<0$ there are two possibilities.
 \begin{itemize}
 \item If $a_1<0$ and $a_2>0$, then $\mathbf z-\mathbf z'=(?,+,-)$ for $\sigma=1$, and $\mathbf z-\mathbf z'=(-,?,+)$ for $\sigma=-1$.
 \item If $a_1>0$ and $a_2<0$, then $\mathbf z-\mathbf z'=(?,-,+)$ for $\sigma=1$, and $\mathbf z-\mathbf z'=(+,?,-)$ for $\sigma=-1$.\qedhere
 \end{itemize}
\end{proof}

\begin{lemma}\label{noesta}
Let $S=\langle n_1,n_2,n_3\rangle$ be a symmetric numerical semigroup with embedding dimension three. With the notation introduced in Section \ref{sec:uno} and Definition \ref{Euclidset}, let $\mathbf z,\mathbf z'\in\mathsf Z(s)$, for some $s\in S\setminus\{0\}$ such that $x=\ell(\mathbf z)-\ell(\mathbf z')\notin \mathrm {Euc}(\delta_1,\delta_2)$. Then there exists $\mathbf z'' \in \mathsf Z(s)$ such that $\ell(\mathbf z')<\ell(\mathbf z'')<\ell(\mathbf z)$.
\end{lemma}
\begin{proof}
If $z_{(i)}\cdot z'_{(i)}\neq 0$, we consider $s'=s-|z_{(i)}-z'_{(i)}|n_i$, which is an element in $S$ with two factorizations $\mathbf z$ and $\mathbf z'$ with $\ell(\mathbf z)-\ell(\mathbf z')=x$ and so that one of them has the $i^{th}$ coordinate equal to zero. With this argument, we can assume  $\mathbf z\cdot \mathbf z'=0$. Later, by adding $\sum_{i=1}^3|z_{(i)}-z'_{(i)}|\mathbf e_i$ ($\mathbf e_i$ is the $i^{th}$ row of the $3\times 3$ identity matrix) to the factorizations obtained, we will recover the original factorizations.

Now, from Proposition \ref{normalizacion}, we can write $\mathbf z-\mathbf z'=a_1\mathbf v_1+\sigma a_2\mathbf v_2+\alpha(\delta_2/g\mathbf v_1-\sigma\delta_1/g\mathbf v_2)$ and  $\mathbf u=a_1\mathbf v_1+\sigma a_2\mathbf v_2$ in such a way that  $\mathbf z-\mathbf z'=(z_{(1)}-z'_{(1)},z_{(2)}-z'_{(2)},z_{(3)}-z'_{(3)})$, $\mathbf u=(u_{(1)},u_{(2)},u_{(3)})$ and $\alpha(\delta_2/g\mathbf v_1-\sigma\delta_1/g\mathbf v_2)=(\alpha_{(1)},\alpha_{(2)},\alpha_{(3)})$. So, we have for $j\in\{1,2,3\}$, $z_{(j)}-z'_{(j)}=u_{(j)}+\alpha_{(j)}$. As the sign of the third coordinates of $\mathbf z-\mathbf z'$ and $\mathbf u$ is the same and the third coordinate of $\mathbf u$ is smaller in absolute value, we necessarily deduce that $\alpha_{(3)}$ has the same sign too. Looking on Table \ref{signs-vx} and Proposition \ref{signos-v-deltas} we have  $\mathrm{sgn}(z_{(i)}-z'_{(i)})=\mathrm{sgn}(u_{(i)})=\mathrm{sgn}(\alpha_{(i)})$ and  $\mathrm{sgn}(z_{(j)}-z'_{(j)})=\mathrm{sgn}(u_{(j)})=\mathrm{sgn}(\alpha_{(j)})$, for some $i\in\{1,2\}$  and $j=3$,.

In light of Proposition \ref{unique}, if $-\delta_2/g<a_1\le 0<a_2\le \delta_1/g$, then $-\delta_1/g<a_2-\delta_1/g\le 0 <a_1+\delta_2/g\le \delta_2/g$. According to the different cases in Corollary \ref{elmenor}, we obtain the following.
\begin{enumerate} 
  \item If $a_1>0$, then $\mathbf v=(v_{(1)},v_{(2)},v_{(3)})=\mathbf w_d$ for $d<\ell(\mathbf z-\mathbf z')=\ell(\mathbf u)$.
  \item If $a_1<0$, then $\mathbf v=(v_{(1)},v_{(2)},v_{(3)})=\mathbf w'_d$ for $d<\ell(\mathbf z-\mathbf z')=\ell(\mathbf u)$.     
\end{enumerate}     

Now we have that $0<\ell(\mathbf v)=d<\ell(\mathbf z)-\ell(\mathbf z')$. By adding $\ell(\mathbf z')$ we obtain $\ell(\mathbf z')<\ell(\mathbf v)+\ell(\mathbf z')<\ell(\mathbf z)$, and it can be easily deduced that $0<\ell(\mathbf z')<\ell(\mathbf z)-\ell(\mathbf v)<\ell(\mathbf z)$.
 
If all the coordinates of $\mathbf v+\mathbf z'$ are positive, or all the coordinates of $\mathbf z-\mathbf v$ are positive, then we would obtain a factorization of $s$ with length comprised between $\ell(\mathbf z')$ and $\ell(\mathbf z)$: thus we only have to prove that either $\mathbf z-\mathbf v$ or $\mathbf v+\mathbf z'$ has positive coordinates. 
 
By Proposition \ref{main} and the behavior of the coordinates described in Corollary \ref{elmenor}, we can deduce that two of the three coordinates (the second and third if $\ell(\mathbf v_2)>0$, and the first and third if $\ell(\mathbf v_2)<0$) of $\mathbf v$ are smaller in absolute value than the corresponding coordinates of $\mathbf u$. Thus we have two possible cases, according to the coordinate whose sign we do not control (the first if $\ell(\mathbf v_2)>0$, and the second if $\ell(\mathbf v_2)<0$).
 
Assume first that this coordinate is negative. In this setting, the corresponding coordinate of  $\mathbf z-\mathbf v$ is positive, while for the other two we have:
  \begin{itemize}
  \item if $v_{(j)}<0$, for some $j$, as $z_{(j)}\ge 0$ we have that the $j^{th}$ coordinate of $\mathbf z-\mathbf v$ will be positive, and
  \item if  $v_{(j)}>0$, we know that $u_{(j)}> v_{(j)}>0$ and, as $\mathrm{sgn}(z_{(j)}-z'_{(j)})=\mathrm{sgn}(u_{(j)})=\mathrm{sgn}(\alpha_{(j)})$ and $z_{(j)}-z'_{(j)}-\alpha_{(j)}=u_{(j)}>v_{(j)}$ we have that $z_{(j)}-v_{(j)}>0$, and thus the $j^{th}$ coordinate $\mathbf z-\mathbf v$ would be again positive.
  \end{itemize}  
 
If the above unknown coordinate is positive, then we take $\mathbf v+\mathbf z'$ instead of $\mathbf z-\mathbf v$, and make the same coordinate positive. For the other two coordinates, if $v_{(j)}>0$, we do nothing, and if $v_{(j)}<0$, as $0<-v_{(j)}<-u_{(j)}=-z_{(j)}+z'_{(j)}+\alpha_{(j)}\le z'_{(j)}$, remember that in this case $\alpha_{(j)}$ is negative. As $z'_{(j)}$ is the $j$-th coordinate of $\mathbf z'$ we would again have that $z'_{(j)}+v_{(j)}$ is a positive coordinate, and by adding $\ell(\mathbf z')$ we obtain that $\ell(\mathbf z')<\ell(\mathbf v+\mathbf z')<\ell(\mathbf z)$. 
\end{proof}

\begin{example}
In the setting of Examples \ref{mainexample-pre} and \ref{mainexample-pre2}, take $11630n_2\in S$ and consider $\mathbf z=(10960,5,715)$ and $\mathbf z'=(0,11603,0)$. We have $\ell(\mathbf z)=11680$ and $\ell(\mathbf z')=11630$, so $\ell(\mathbf z -\mathbf z')=50<\max\{\delta_1,\delta_2\}$. As $z_{(2)}\cdot z'_{(2)}\neq 0$, we can consider $s'=11630n_2-5n_2$ as the $s'$ in the proof of Lemma \ref{noesta}. So, we are going to work with $s'=11625n_2$, $\mathbf z=(10960,0,715)$ and $\mathbf z'=(0,11625,0)$, obtaining $\mathbf z-\mathbf z'=(10960,-11625,715)=20\mathbf v_1-5\mathbf v_2$. In this case $\alpha=0,$ whence $\mathbf u=20\mathbf v_1-5\mathbf v_2$. As we showed in Example \ref{mainexample-pre2}, we can consider $\mathbf w_d=\mathbf w_{43}=3\mathbf v_1-8\mathbf v_2=(1644,-1705,104)$. Now, if we consider $\mathbf w_d+\mathbf z'=(1644,9920,104)$, we obtain a new factorization of $11625n_2$ with $11630=\ell(\mathbf z')<\ell(\mathbf w_d+\mathbf z')=11668<\ell(\mathbf z)=11680$, while if we consider $\mathbf z-\mathbf w_d=(9316,1705,611)$, a new factorization of $11625n_2$ can be obtained, with $\ell(\mathbf z-\mathbf w_d)=11632$ which, again, is between $\ell(\mathbf z')$ and $\ell(\mathbf z)$. In this case, both factorizations can be chosen. Finally adding $5$ to the second coordinate, we obtain factorizations of the original element, $11630n_2$.

Now if we choose $71300n_2\in S$, we have two factorizations $\mathbf z=(66856,0,4394)$ and $\mathbf z'=(0,71300,0)$ whose lengths are $\ell(\mathbf z)=71250$ and $\ell(\mathbf z')=71300$, obtaining again that $\ell(\mathbf z')-\ell(\mathbf z)=50$. In this case $\mathbf u=-122\mathbf v_1+338\mathbf v_2$ because $\mathbf z'-\mathbf z$ is of the type $(?,+,-)$ so $a_1<0$ and $a_2>0$. According to Example \ref{mainexample-pre2}, we can pick $\mathbf w'_{33}=-\mathbf v_1+3\mathbf v_2=(-548,620,-39)$ to obtain $\mathbf w'_d+\mathbf z=(66308,620,4355)$ and $\mathbf z'-\mathbf w'_d=(548,70680,39)$; hence we obtained two factorizations of $71300n_2$ whose lengths $\ell(\mathbf w'_d+\mathbf z)=71283$ and $\ell(\mathbf z'-\mathbf w'_d)=71267$ are both between $\ell(\mathbf z)$ and $\ell(\mathbf z')$.
\end{example}

\begin{lemma}\label{siesta}
Let $S=\langle n_1,n_2,n_3\rangle$ be a symmetric numerical semigroup with embedding dimension three. With the notation introduced in Section \ref{sec:uno} and Definition \ref{Euclidset}, consider $d\in \mathrm{Euc}(\delta_1,\delta_2)$. Then there exists $s\in S$ and $\mathbf z,\mathbf z'\in\mathsf Z(s)$, with $d=\ell(\mathbf z)-\ell(\mathbf z')$, and such that for any other $\mathbf z''\in\mathsf Z(s)$ either $\ell(\mathbf z'')\le \ell(\mathbf z')$ or $\ell(\mathbf z'')\ge\ell(\mathbf z)$.
\end{lemma}
\begin{proof}
Let $d\in \mathrm{Euc}(\delta_1,\delta_2)$, and set $\mathbf v=x_1\mathbf v_1+\sigma x_2\mathbf v_2$, with $x_1x_2<0$, and $-\delta_1/g<x_2\le \delta_1/g$ and $-\delta_2/g<x_1\le \delta_2/g$, such that $\ell(\mathbf v)=d$. Then
\[
\mathbf v=(x_1v_{1_{(1)}}+\sigma x_2v_{2_{(1)}},x_1v_{1_{(2)}}+\sigma x_2v_{2_{(2)}},-\sigma x_2v_{2_{(3)}}), 
\]
with $\ell(\mathbf v)=d>0$. Clearly, $\mathbf v$ has two coordinates with the same sign, and the other with opposite sign. Let us denote the latter by  $i$. 

Set $s=|v_{x_{(i)}}|n_i$, and $\mathbf z=(v^+_{x_{(1)}},v^+_{x_{(2)}},v^+_{x_{(3)}})$ and $\mathbf z'=(v^-_{x_{(1)}},v^-_{x_{(2)}},v^-_{x_{(3)}})$, where $a^+=(|a|+a)/2\ge 0$ and $a^-=(|a|-a)/2\ge 0$. Then $\mathbf z$ and $\mathbf z'$ are factorizations of $s$. The rest of the factorizations of $s$ are of the form $\mathbf z''=\mathbf z'+a_1\mathbf v_1+a_2\sigma\mathbf v_2$, with $a_1a_2<0$. We are concerned with those such that $0<a_1\delta_1+a_2\delta_2<d$, in order to find one with length between $\ell(\mathbf z)$ and $\ell(\mathbf z')$. Denote $\mathbf u=a_1\mathbf v_1+a_2\sigma\mathbf v_2=(u_{(1)},u_{(2)},u_{(3)})$. 

Next, we prove that $|a_1|\ge |x_1|$ and $|a_2|>|x_2|$.

Assume that $\ell(\mathbf u)\in \mathrm{Euc}(\delta_1,\delta_2)$. Then, as $\ell(\mathbf u)<\ell(\mathbf v)$, by construction (see Remark \ref{nota}) $|a_1|\ge |x_1|$ and $|a_2|>|x_2|$. On the other hand, if $\ell(\mathbf u)\notin \mathrm{Euc}(\delta_1,\delta_2)$, then we consider $e=\mathrm {bsm}(\ell(\mathbf u))\in \mathrm {Euc}(\delta_1,\delta_2)$ associated to $\mathbf u$ (Proposition \ref{main}), and then we have $|a_1|>|e_1|>|x_1|$ and $|a_2|>|e_2|>|x_2|$, and the assertion is proved.
 
Thus, Table  \ref{signs-vu} collects the different possible settings.

\begin{table}[h]
\begin{center}
\begin{tabular}{c|cc}
                  & $\sigma=1$   & $\sigma=-1$ \\
                  & $\mathbf v_1=(+,-,0)$, $\mathbf v_2=(+,+,-)$  & $\mathbf v_1=(+,-,0)$, $\sigma\mathbf v_2=(-,-,+)$ \\
\hline
$x_1>0$, $x_2<0$  & $\mathbf v=(?,-,+)$ & $\mathbf v=(+,?,-)$ \\
$x'_1<0$, $x'_2>0$  & $\mathbf v=(?,+,-)$ & $\mathbf v=(-,?,+)$ \\
$a_1>0$, $a_2<0$  & $\mathbf u=(?,-,+)$ & $\mathbf u=(+,?,-)$ \\
$a'_1<0$, $a'_2>0$  & $\mathbf u=(?,+,-)$ & $\mathbf u=(-,?,+)$ \\
\end{tabular}
\end{center}
\caption{Signs for $\mathbf v$ and $\mathbf u$}
\label{signs-vu}
\end{table}

According to Table \ref{signs-vu}, we have the following cases.
\begin{enumerate}
	\item If $\sigma=1$, we obtain that $|u_{(2)}|>|v_{(2)}|$ and $|u_{(3)}|>|v_{(3)}|$
  \begin{enumerate}
  	\item If $\mathbf v=(?,-,+)$, then $z_{(2)}=0$, and $z'_{(2)}\le 0$. If in addition $u_{(2)}<0$, then  $z''_{(2)}<0$, which is a contradiction; and if  $u_{(2)}>0$ we derive $u_{(3)}<0$ and as $|u_3|=|a_2||v_{2_{(3)}}|>|x_2||v_{2_{(3)}}|=|v_{(3)}|$, we obtain $z''_{(3)}<0$, which is another contradiction.
  	\item If $\mathbf v=(?,+,-)$, then $z_{(3)}=0$. If $u_{(3)}<0$, then $z''_{(3)}<0$, a contradiction. If $u_{(3)}>0$, we have that $u_{(2)}<0$ and as $\mathrm{sgn}(a_1v_{1_{(2)}})=\mathrm{sgn}(a_2v_{2_{(2)}})$, we can assure $|u_{(2)}|=|a_1||v_{1_{(2)}}|+|a_2||v_{2_{(2)}}|>|x_1||v_{1_{(2)}}|+|x_2||v_{2_{(2)}}|=|v_{(2)}|$. So we get $z''_{(2)}<0$, which is again a contradiction.
  \end{enumerate}
    \item If $\sigma=-1$, we have that $|u_{(1)}|>|v_{(1)}|$ and $|u_{(3)}|>|v_{(3)}|$.
  \begin{enumerate}   
   \item If $\mathbf v=(+,?,-)$, then $z_{(3)}=0$. The case $u_{(3)}<0$ leads to $z''_{(3)}<0$, which is a contradiction; while from $u_{(3)}>0$ we deduce that $u_{(1)}<0$, and as $\mathrm{sgn}(a_1v_{1_{(1)}})=\mathrm{sgn}(\sigma a_2v_{2_{(1)}})$, we can assure $|u_{(1)}|=|a_1||v_{1_{(1)}}|+|a_2||\sigma v_{2_{(1)}}|>|x_1||v_{1_{(1)}}|+|x_2||\sigma v_{2_{(1)}}|=|v_{(1)}|$ we obtain $z''_{(1)}<0$, yielding once more a contradiction.
   \item If $\mathbf v=(-,?,+)$, then $z_{(1)}=0$. If the inequality $u_{(1)}<0$ holds, then $z''_{(1)}<0$, yielding a contradiction. If  $u_{(1)}>0$, then $u_{(3)}<0$, and as  $|u_{(3)}|=|a_2||\sigma v_{2_{(3)}}|>|x_2||\sigma v_{2_{(3)}}|=|v_{(3)}|$ we derive $z''_{(3)}<0$, which is also a contradiction.
  \end{enumerate}
\end{enumerate}

This proves that there is no factorization of $s$ with length between  $\ell(\mathbf z)$ and $\ell(\mathbf z')$.  
\end{proof}

\end{document}